\documentclass[a4paper,11pt]{amsart}
\usepackage[top=1.25in, bottom=1.25in, left=1.25in, right=1.25in]{geometry}
\usepackage{graphicx}
\usepackage{amsfonts}
\usepackage{epsf}
\usepackage{amssymb}
\usepackage{amsmath}
\usepackage{amscd}
\usepackage{amsthm}
\usepackage{tikz}
\usetikzlibrary{
  cd,
  calc,
  positioning,
  arrows,
  decorations.pathreplacing,
  decorations.markings,
}
\usepackage{pdfpages}
\usepackage{fancyhdr}
\usepackage{setspace}
\usepackage{hyperref}
\usepackage[all]{xy}
\usetikzlibrary{matrix}
\usepackage{verbatim}
\usepackage{enumerate}
\usepackage{mathrsfs}
\usepackage{pinlabel}
\usepackage{lscape}
\usepackage{color}

\newtheorem{theorem}{Theorem}[section]

\newtheorem{proposition}[theorem]{Proposition}

\theoremstyle{definition}

\newtheorem{remark}[theorem]{Remark}

\newtheorem*{case2'}{Case 2$'$}
\newtheorem{theorem-named}{}
\newtheorem{theorem-labeled}{Theorem}
\newtheorem{definition-named}{}
\newtheorem{conjecture-named}{}
\newtheorem{case-named}{}

\numberwithin{equation}{section}

\newcommand{\tors}{{\rm Tors}}

\newcommand{\Z}{\mathbb{Z}}

\def\Z{\mathbb{Z}}

\newcommand{\Id}{\operatorname{Id}}

\newcommand{\rk}{\operatorname{rk}}

\makeatletter

\begin{document}
\title{Embedding lens spaces in definite 4-manifolds}

\author{Paolo Aceto}
\address{Mathematical Institute University of Oxford, Oxford, United Kingdom}
\email{paoloaceto@gmail.com}
\urladdr{www.maths.ox.ac.uk/people/paolo.aceto}
\author{JungHwan Park}
\address{School of Mathematics, Georgia Institute of Technology, Atlanta, GA, USA}
\email{junghwan.park@math.gatech.edu }
\urladdr{people.math.gatech.edu/$\sim$jpark929/}
\def\subjclassname{\textup{2010} Mathematics Subject Classification}
\expandafter\let\csname subjclassname@1991\endcsname=\subjclassname
\expandafter\let\csname subjclassname@2000\endcsname=\subjclassname

\begin{abstract}
Every lens space has a locally flat embedding in a connected sum of 8 copies of the complex projective plane and a smooth embedding in $n$ copies of the complex projective plane for some positive integer $n$. We show that there is no $n$ such that every lens space smoothly embeds in $n$ copies of the complex projective plane.
\end{abstract}

\maketitle

\section{Introduction}\label{sec:intro}

Every closed $3$-manifold embeds in $S^5$ \cite{Hirsch:1961-1,Rohlin:1965-1,Wall:1965-1}. On the other hand, there are  strong restrictions on which closed $3$-manifolds can embed in $S^4$. It was shown by Hantzsche \cite{Hantzsche:1938-1} that if a rational homology $3$-sphere $Y$ embeds in $S^4$, then $H_1(Y;\mathbb{Z}) \cong G \oplus G$. In particular, no lens space (other than $S^3$ and $S^1 \times S^2$) embed in $S^4$. Further, a punctured lens space  $L(p,q)_0$ admits an embedding in $S^4$ if and only if $p$ is odd \cite{Zeeman:1965-1,Epstein:1965-1}. Also, we have a complete understanding of which connected sums of lens spaces can be smoothly embedded in $S^4$ by Donald \cite{Donald:2015-1} (see also \cite{Kawauchi-Kojima:1980-1,Gilmer-Livingston:1983-1,Fintushel-Stern:1987-1}). There are also various interesting results on embedding other $3$-manifolds in $S^4$ \cite{Kawauchi:1977-1, Casson-Harer:1981-1, Hillman:1996-1, Crisp-Hillman:1998-1, Budney-Burton:2008-1,Hillman:2009-1,Issa-McCoy:2018-1}.

Even though the embedding problem of lens spaces in $S^4$ is completely solved, there are many interesting generalizations. In this paper, we focus on the embedding problem of lens spaces in definite $4$-manifolds (the embedding problem of lens spaces in spin $4$-manifolds has been studied in \cite{Aceto-Golla-Larson:2017-1}). In  \cite{Edmonds-Livingston:1996-1}, Edmonds and Livingston showed that every lens space smoothly embeds in $\#_n\mathbb{C}\mathbb{P}^2$ for some positive integer $n$. Further, they showed that there is a family of lens spaces that do not have locally flat embeddings in $\#_4\mathbb{C}\mathbb{P}^2$. Later, Edmonds \cite{Edmonds:2005-1} showed that every lens space has a locally flat embedding in $\#_8\mathbb{C}\mathbb{P}^2$ using independent works of Boyer \cite{Boyer:1993-1} and Stong \cite{Stong:1993-1} which extend Freedman’s \cite{Freedman:1982-1} realization result. In contrast, we show that there is no $n$ such that every lens space smoothly embeds in $\#_n\mathbb{C}\mathbb{P}^2$. Our main argument relies on  Donaldson's diagonalization theorem~\cite{Donaldson:1987-1} and is based on the combinatorics of integral lattices.

\begin{theorem}\label{thm:main} Let $L(p,q)$ be a lens space bounded by the canonical positive definite plumbed manifold $P(p,q)$ with the plumbing graph
$$
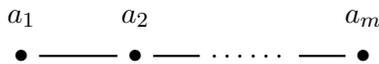
\begin{tikzpicture}[xscale=1.5,yscale=-0.5]
\node (A0_1) at (1, 0) {$a_1$};
\node (A0_2) at (2, 0) {$a_2$};
\node (A0_4) at (4, 0) {$a_m$};

\node (A1_1) at (1, 1) {$\bullet$};
\node (A1_2) at (2, 1) {$\bullet$};
\node (A1_3) at (3, 1) {$\dots\dots$};
\node (A1_4) at (4, 1) {$\bullet$};
\path (A1_2) edge [thick,-] node [auto] {$\scriptstyle{}$} (A1_3);
\path (A1_3) edge [thick,-] node [auto] {$\scriptstyle{}$} (A1_4);
\path (A1_1) edge [thick,-] node [auto] {$\scriptstyle{}$} (A1_2);
\end{tikzpicture}
$$
Suppose that $a_i \geq 6 $ for all $i$. If $L(p,q)$ smoothly embeds in a definite $4$-manifold $W$ with $b_1(W)=0$, then $b_2(W)> m$. In particular, if $L(p,q)$ smoothly embeds in $\#_n\mathbb{C}\mathbb{P}^2$, then $n > m$.\end{theorem}

Furthermore, we show that if a lens space $L(p,q)$ as in Theorem~\ref{thm:main} bounds a smooth, positive definite $4$-manifold, there is a strong restriction on its intersection form. This also reflects the big discrepancy between  the smooth and the topological category in dimension $4$ since for every $3$-manifold $Y$ and a $\mathbb{Z}$-valued symmetric bilinear form $Q$ that presents the linking form of $Y$, $Q$ is realized as the intersection form of a simply connected, topological $4$-manifold bounded by $Y$ \cite{Boyer:1993-1, Stong:1993-1}. For instance, every lens space bounds a simply connected, positive definite, topological $4$-manifold $X$ with $b_2(X) \leq 6$ \cite{Edmonds:2005-1}. Recall that an \emph{integral lattice} is a pair $(G,Q)$, where $G$ is a finitely generated free abelian group and $Q$ is a $\mathbb{Z}$-valued symmetric bilinear form defined on $G$. The integral lattice with the standard positive definite form is denoted by $(\mathbb{Z}^N, \Id)$. A \emph{morphism of integral lattices} is a homomorphism of abelian groups which preserves the form. An \emph{embedding} is an injective morphism. To a given $4$-manifold $X$, we associate the integral lattice $(H_2(X;\mathbb{Z})/\tors,Q_X)$, where  $Q_X$ is the intersection form on $X$.

\begin{theorem}\label{thm:2} Let $L(p,q)$ be a lens space that satisfies the assumption of Theorem~\ref{thm:main}. If $L(p,q)$ bounds a smooth, positive definite $4$-manifold $X$, then $b_2(X)\geq m$ and there is an embedding $$
(H_2(X;\mathbb{Z})/\tors,Q_X) \hookrightarrow (H_2(P(p,q);\mathbb{Z}),Q_{P(p,q)}) \oplus (\mathbb{Z}^{b_2(X)- m}, \Id).
$$
Moreover, $b_2(X) = m$ if and only if $X$ and $P(p,q)$ have isomorphic intersection forms.
\end{theorem}

\subsection*{Acknowledgments}
This project started when both authors were at the Max-Planck-Institut f\"{u}r Mathematik. We would like to thank the MPIM for providing an excellent environment for research.
 

\section{proof of Theorem $1.1$ and $1.2$}\label{sec:proof}

We work in the smooth category and all manifolds are oriented. Recall that the lens space $L(p,q)$ is the result of $-p/q$ Dehn surgery on the unknot. Up to orientation preserving diffeomorphism, we may assume that $p > q >0$. For the rest of this article we only consider lens spaces $L(p,q)$ that bound the canonical positive definite plumbded manifolds $P(p,q)$ with the plumbing graph $\Gamma_{p,q}:=$
$$
\begin{tikzpicture}[xscale=1.5,yscale=-0.5]
\node (A0_1) at (1, 0) {$a_1$};
\node (A0_2) at (2, 0) {$a_2$};
\node (A0_4) at (4, 0) {$a_m$};

\node (A1_1) at (1, 1) {$\bullet$};
\node (A1_2) at (2, 1) {$\bullet$};
\node (A1_3) at (3, 1) {$\dots\dots$};
\node (A1_4) at (4, 1) {$\bullet$};
\path (A1_2) edge [thick,-] node [auto] {$\scriptstyle{}$} (A1_3);
\path (A1_3) edge [thick,-] node [auto] {$\scriptstyle{}$} (A1_4);
\path (A1_1) edge [thick,-] node [auto] {$\scriptstyle{}$} (A1_2);
\end{tikzpicture}
$$
where $a_i \geq 6 $ for all $i$.

Let $-L(p,q)$ be the lens space $L(p,q)$ with the reversed orientation, then there is an orientation preserving diffeomorphism between $-L(p, q)$ and $L(p, p-q)$. Using Riemenschneider’s point rule \cite{Riemenschneider:1974-1} (see also \cite{Lisca:2007-1,Lecuona:2012-1}), we see that $L(p,p-q)$ bounds a canonical positive definite plumbed manifold $P(p,p-q)$ with the plumbing graph $\Gamma_{p,p-q}:=$
$$
\begin{tikzpicture}[xscale=1.5,yscale=-0.5]
\node (A0_1) at (1, 0) {$2$};

\node (A0_3) at (2, 0) {$2$};
\node (A0_4) at (2.5, 0) {$3$};
\node (A0_5) at (3, 0) {$2$};

\node (A0_7) at (4, 0) {$2$};
\node (A0_8) at (4.5, 0) {$3$};
\node (A0_9) at (5, 0) {$2$};

\node (A0_11) at (6, 0) {$2$};
\node (A0_12) at (6.5, 0) {$3$};
\node (A0_13) at (7, 0) {$2$};

\node (A0_14) at (8, 0) {$2$};
\node (A0_15) at (8.5, 0) {$3$};
\node (A0_16) at (9, 0) {$2$};

\node (A0_16) at (10, 0) {$2$};

\node (A1_1) at (1, 1) {$\bullet$};
\node (A1_2) at (1.5, 1) {$\dots$};
\node (A1_3) at (2, 1) {$\bullet$};
\node (A1_4) at (2.5, 1) {$\bullet$};
\node (A1_5) at (3, 1) {$\bullet$};
\node (A1_6) at (3.5, 1) {$\dots$};
\node (A1_7) at (4, 1) {$\bullet$};
\node (A1_8) at (4.5, 1) {$\bullet$};

\node (A1_9) at (5, 1) {$\bullet$};
\node (A1_10) at (5.5, 1) {$\dots$};
\node (A1_11) at (6, 1) {$\bullet$};
\node (A1_12) at (6.5, 1) {$\bullet$};
\node (A1_13) at (7, 1) {$\bullet$};
\node (A1_14) at (7.5, 1) {$\dots$};
\node (A1_15) at (8, 1) {$\bullet$};
\node (A1_16) at (8.5, 1) {$\bullet$};
\node (A1_17) at (9, 1) {$\bullet$};
\node (A1_18) at (9.5, 1) {$\dots$};
\node (A1_19) at (10, 1) {$\bullet$};

\path (A1_1) edge [thick,-] node [auto] {$\scriptstyle{}$} (A1_2);
\path (A1_2) edge [thick,-] node [auto] {$\scriptstyle{}$} (A1_3);
\path (A1_3) edge [thick,-] node [auto] {$\scriptstyle{}$} (A1_4);
\path (A1_4) edge [thick,-] node [auto] {$\scriptstyle{}$} (A1_5);
\path (A1_5) edge [thick,-]node [auto] {$\scriptstyle{}$} (A1_6);
\path (A1_6) edge [thick,-] node [auto] {$\scriptstyle{}$} (A1_7);
\path (A1_7) edge [thick,-] node [auto] {$\scriptstyle{}$} (A1_8);
\path (A1_8) edge [thick,-] node [auto] {$\scriptstyle{}$} (A1_9);
\path (A1_9) edge [thick,-] node [auto] {$\scriptstyle{}$} (A1_10);
\path (A1_10) edge [thick,-] node [auto] {$\scriptstyle{}$} (A1_11);
\path (A1_11) edge [thick,-] node [auto] {$\scriptstyle{}$} (A1_12);
\path (A1_12) edge [thick,-] node [auto] {$\scriptstyle{}$} (A1_13);
\path (A1_13) edge [thick,-] node [auto] {$\scriptstyle{}$} (A1_14);
\path (A1_14) edge [thick,-] node [auto] {$\scriptstyle{}$} (A1_15);
\path (A1_15) edge [thick,-] node [auto] {$\scriptstyle{}$} (A1_16);
\path (A1_16) edge [thick,-] node [auto] {$\scriptstyle{}$} (A1_17);
\path (A1_17) edge [thick,-] node [auto] {$\scriptstyle{}$} (A1_18);
\path (A1_18) edge [thick,-] node [auto] {$\scriptstyle{}$} (A1_19);

\node (A2_1) at (1.5, 2){$\underbrace{\hphantom{----}}_{a_1-2}$};
\node (A2_2) at (3.5, 2){$\underbrace{\hphantom{----}}_{a_2-3}$};
\node (A2_3) at (7.5, 2){$\underbrace{\hphantom{----}}_{a_{m-1}-3}$};
\node (A2_4) at (9.5, 2){$\underbrace{\hphantom{----}}_{a_{m}-2}$};
\end{tikzpicture}
$$

We denote the integral lattice associated to $P(p,q)$ as $(\Z\Gamma_{p,q},Q_{p,q})$ and call it the \emph{integral lattice associated with} $L(p,q)$. Similarly, we also have a \emph{dual} positive definite integral lattice $(\Z\Gamma_{p,p-q},Q_{p,p-q})$ associated with $L(p,p-q)$.

\begin{proposition}\label{prop:dualembedding} If there is an embedding from $(\Z\Gamma_{p,p-q},Q_{p,p-q})$ to $(\mathbb{Z}^N, \Id)$, then $N\geq m + \rk(\Z\Gamma_{p,p-q})$.
\end{proposition}
\begin{proof}We label the first $a_1$ vertices of $\Gamma_{p,p-q}$ as follows

$$
\begin{tikzpicture}[xscale=1.5,yscale=-0.5]
\node (A0_1) at (1, 0) {$2$};
\node (A0_3) at (3, 0) {$2$};
\node (A0_4) at (4, 0) {$3$};
\node (A0_5) at (5, 0) {$2$};

\node (A1_1) at (1, 1) {$\bullet$};
\node (A1_2) at (2, 1) {$\dots\dots$};
\node (A1_3) at (3, 1) {$\bullet$};
\node (A1_4) at (4, 1) {$\bullet$};
\node (A1_5) at (5, 1) {$\bullet$};

\node (A2_1) at (1, 2) {$x_1$};
\node (A2_3) at (3, 2) {$x_{a_1-2}$};
\node (A2_4) at (4, 2) {$x_{a_1-1}$};
\node (A2_5) at (5, 2) {$x_{a_1}$};

\path (A1_1) edge [thick,-] node [auto] {$\scriptstyle{}$} (A1_2);
\path (A1_2) edge [thick,-] node [auto] {$\scriptstyle{}$} (A1_3);
\path (A1_3) edge [thick,-] node [auto] {$\scriptstyle{}$} (A1_4);
\path (A1_4) edge [thick,-] node [auto] {$\scriptstyle{}$} (A1_5);
\end{tikzpicture}
$$
Let $\{e_1,\ldots, e_N\}$ be the standard basis for $(\mathbb{Z}^N, \Id)$. By abuse of notation we identify $(\Z\Gamma_{p,p-q},Q_{p,p-q})$ with its image in the standard lattice. It is straightforward to see that a chain of $2$'s with length longer than $3$ has a unique embedding. Hence up to reordering and changing sign of the standard basis elements we may write $$x_i = e_i + e_{i+1} \text{ for } 1\leq i \leq a_1-2.$$ Further, since $x_{a_1-1}$ intersects with $x_{a_1-2}$ once and has norm $3$, $$x_{a_1-1} = e_{a_1-1} + {\sum\limits_{a_1-1<j} c_j e_j}.$$ Lastly, $x_{a_1}$ has a trivial intersection with $e_{a_1-1}$ since it is disjoint from the first chain of $2$'s and it has norm $2$. Therefore, $x_{a_1-1} - e_{a_1-1}$ is disjoint from the first chain of $2$'s and intersects $x_{a_1}$ once. Now, if we only consider $x_{a_1-1} - e_{a_1-1}$ and all the vertices that reside on the right hand side of $x_{a_1-1}$, we get to the same situation as we have started with. Hence we can repeat the same argument to get the following identifications
$$
\begin{tikzpicture}[xscale=1.5,yscale=-0.5]
\node (A0_1) at (1, 0) {$2$};
\node (A0_2) at (3, 0) {$2$};
\node (A0_4) at (7, 0) {$2$};

\node (A1_1) at (1, 1) {$\bullet$};
\node (A1_2) at (3, 1) {$\bullet$};
\node (A1_3) at (5, 1) {$\dots\dots\dots\dots$};
\node (A1_4) at (7, 1) {$\bullet$};

\node (A2_1) at (1, 2) {$e_{n_\ell}+e_{n_\ell+1}$};
\node (A2_2) at (3, 2) {$e_{n_\ell+1}+e_{n_\ell+2}$};
\node (A2_4) at (7, 2) {$e_{n_\ell+a_\ell-4}+e_{n_\ell+a_\ell-3}$};

\path (A1_2) edge [thick,-] node [auto] {$\scriptstyle{}$} (A1_3);
\path (A1_3) edge [thick,-] node [auto] {$\scriptstyle{}$} (A1_4);
\path (A1_1) edge [thick,-] node [auto] {$\scriptstyle{}$} (A1_2);
\end{tikzpicture}
$$
for each chain of $2$'s where $n_\ell = \sum\limits_{k=1}^{\ell-1} a_k + (3-\ell)$ and $2\leq \ell \leq m-1$, and
$$
\begin{tikzpicture}[xscale=1.5,yscale=-0.5]
\node (A0_1) at (1, 0) {$2$};
\node (A0_2) at (3, 0) {$2$};
\node (A0_4) at (7, 0) {$2$};

\node (A1_1) at (1, 1) {$\bullet$};
\node (A1_2) at (3, 1) {$\bullet$};
\node (A1_3) at (5, 1) {$\dots\dots\dots\dots$};
\node (A1_4) at (7, 1) {$\bullet$};

\node (A2_1) at (1, 2) {$e_{n_m}+e_{n_m+1}$};
\node (A2_2) at (3, 2) {$e_{n_m+1}+e_{n_m+2}$};
\node (A2_4) at (7, 2) {$e_{n_m+a_m-3}+e_{n_m+a_m-2}$};

\path (A1_2) edge [thick,-] node [auto] {$\scriptstyle{}$} (A1_3);
\path (A1_3) edge [thick,-] node [auto] {$\scriptstyle{}$} (A1_4);
\path (A1_1) edge [thick,-] node [auto] {$\scriptstyle{}$} (A1_2);
\end{tikzpicture}
$$ for the $m$-th chain of $2$'s where $n_m = \sum\limits_{k=1}^{m-1} a_k + (3-m)$. The $\ell$-th vertex with weight $3$ shares one coordinate from its left chain and one coordinate from its right chain. Further, it needs an extra coordinate, which we denote it by $e_{n_{\ell+1}-1}$. 

In total, we have used $\sum\limits_{k=1}^{m} a_k - m +1$ coordinates which implies that $N\geq \sum\limits_{k=1}^{m} a_k - m +1 $. The proof is complete by observing that $m + \rk(\Z\Gamma_{p,p-q}) = \sum\limits_{k=1}^{m} a_k - m +1$.\end{proof}

\begin{remark}\label{rmk:unique} In fact, the proof of Proposition~\ref{prop:dualembedding} shows that there is a \emph{unique} embedding up to change of basis from $(\Z\Gamma_{p,p-q},Q_{p,p-q})$ to $(\mathbb{Z}^N, \Id)$ when $N \geq m + \rk(\Z\Gamma_{p,p-q})$.
\end{remark}

\begin{proposition}\label{prop:lowerbound} If $L(p,q)$ bounds a positive definite $4$-manifold $X$, then $b_2(X)\geq m$.\end{proposition} 
\begin{proof} Let $W$ be the closed $4$-manifold obtained by gluing $X$ with $P(p,p-q)$ along $L(p,q)$. 
We obtain the following embedding 
$$
(H_2(X;\mathbb{Z})/\tors,Q_X) \oplus (\Z\Gamma_{p,p-q},Q_{p,p-q}) \hookrightarrow (H_2(W;\mathbb{Z})/\tors,Q_W).
$$ Further, by Donaldson's diagonalization theorem~\cite{Donaldson:1987-1} we have $$(H_2(W;\mathbb{Z})/\tors,Q_W) \cong(\mathbb{Z}^{b_2(W)}, \Id).$$ Combining Proposition~\ref{prop:dualembedding} with $b_2(W)=b_2(X) +\rk(\Z\Gamma_{p,p-q})$ completes the proof.
\end{proof}

\begin{proof}[Proof of Theorem~\ref{thm:main}] Suppose $L(p,q)$ smoothly embeds in a definite $4$-manifold $W$. Since $L(p,q)$ embeds in $W$ if and only if $L(p,q)$ embeds is $-W$, we may assume that $W$ is positive definite. Then $L(p,q)$ separates $W$ into two positive definite components, the closures of which we denote by $X_1$ and $X_2$. Note that by the Mayer–Vietoris sequence we have $b_2(X_1)+b_2(X_2) = b_2(W)$ and the result follows from Proposition~\ref{prop:lowerbound} and the fact that $L(p,q)$ does not bound a rational ball (see \cite{Lisca:2007-1}).
\end{proof}

The rest of the section is devoted to proving Theorem~\ref{thm:2}. Suppose there is an embedding of an integral lattice $\phi\colon \Gamma\hookrightarrow \Gamma'$. Then the orthogonal complement of $\Gamma$ in $\Gamma'$ with respect to $\phi$ is defined as follows, $$ \Gamma^\perp_\phi:= \{x \in \Gamma' \mid x\cdot \phi(y) =0 \text{ for all } y \in \Gamma\}.$$

\begin{proposition}~\label{prop:ortho} Suppose there is an embedding $$\phi\colon (\Z\Gamma_{p,p-q},Q_{p,p-q})\hookrightarrow (\mathbb{Z}^N, \Id),$$
then $(\Z\Gamma_{p,p-q},Q_{p,p-q})^\perp_\phi \cong (\Z\Gamma_{p,q},Q_{p,q}) \oplus(\mathbb{Z}^{N-m-\rk\left(\Z\Gamma_{p,p-q}\right)}, \Id)$. In particular, if $N=m + \rk(\Z\Gamma_{p,p-q})$, then $(\Z\Gamma_{p,p-q},Q_{p,p-q})^\perp_\phi \cong (\Z\Gamma_{p,q},Q_{p,q})$.
\end{proposition}

\begin{proof} 
From Proposition~\ref{prop:dualembedding} and Remark~\ref{rmk:unique}, we know that there is a unique embedding up to change of basis from $(\Z\Gamma_{p,p-q},Q_{p,p-q})$ to $(\mathbb{Z}^N, \Id)$. Hence we may decompose $\phi$ as follows
$$\phi\colon (\Z\Gamma_{p,p-q},Q_{p,p-q})\hookrightarrow (\mathbb{Z}^{m+\rk\left(\Z\Gamma_{p,p-q}\right)}, \Id)\oplus(\mathbb{Z}^{N-m-\rk\left(\Z\Gamma_{p,p-q}\right)}, \Id),$$ where the image of $\phi$ on the second summand is trivial. Let $\pi$ be the projection map from $(\mathbb{Z}^N, \Id)$ to $(\mathbb{Z}^{m+\rk\left(\Z\Gamma_{p,p-q}\right)}, \Id)$, then we have the following identification \begin{equation}\label{eqn:first}
(\Z\Gamma_{p,p-q},Q_{p,p-q})^\perp_{\phi} \cong (\Z\Gamma_{p,p-q},Q_{p,p-q})^\perp_{\pi \circ \phi} \oplus (\mathbb{Z}^{N-m-\rk\left(\Z\Gamma_{p,p-q}\right)}, \Id).
\end{equation}

Let $W$ be the closed $4$-manifold obtained by gluing $P(p,q)$ with $P(p,p-q)$ along $L(p,q)$. Using Donaldson's diagonalization theorem~\cite{Donaldson:1987-1}, we have an embedding 
$$\psi\colon
(\Z\Gamma_{p,q},Q_{p,q}) \oplus (\Z\Gamma_{p,p-q},Q_{p,p-q}) \hookrightarrow (\mathbb{Z}^{m+\rk\left(\Z\Gamma_{p,p-q}\right)}, \Id).
$$ Again, since there is a unique embedding up to change of basis from $(\Z\Gamma_{p,p-q},Q_{p,p-q})$ to $(\mathbb{Z}^{m+\rk\left(\Z\Gamma_{p,p-q}\right)}, \Id)$, we may assume that the embedding $\psi$ restricted to $(\Z\Gamma_{p,p-q},Q_{p,p-q})$, denoted by $\psi_{p,p-q}$, coincides with $\pi \circ \phi$. In particular, 
\begin{equation}\label{eqn:second}(\Z\Gamma_{p,p-q},Q_{p,p-q})^\perp_{\pi \circ \phi} \cong (\Z\Gamma_{p,p-q},Q_{p,p-q})^\perp_{\psi_{p,p-q}}.
\end{equation} Further, we can use the coordinates from the proof of Proposition~\ref{prop:dualembedding}. 

By restricting $\psi$ to $\Z\Gamma_{p,q}$, denoted by $\psi_{p,q}$, we have $$\psi_{p,q}\colon(\Z\Gamma_{p,q},Q_{p,q}) \hookrightarrow (\Z\Gamma_{p,p-q},Q_{p,p-q})^\perp_{\psi_{p,p-q}}.$$ Now, suppose $x= \sum\limits_{i=1}^{N} c_i e_i  \in (\Z\Gamma_{p,p-q},Q_{p,p-q})^\perp_{\psi_{p,p-q}}$. Since $x$ needs to have trivial intersections with all the chain of $2$'s, we have 
$$c_{1} = \cdots = c_{a_1-1}, c_{n_\ell} = \cdots = c_{n_\ell+a_\ell-3}, \text{ for } 2 \leq \ell \leq m-1, \text{ and } c_{n_m} = \cdots = c_{n_m+a_m-2}$$
where $n_\ell = \sum\limits_{\ell=1}^{i-1} a_\ell + (3-\ell)$ for $2\leq \ell \leq m$. Further, $x$ has trivial intersections with vertices with weight $3$. This implies $$c_{n_{\ell+1}-2}+c_{n_{\ell+1}-1}+c_{n_{\ell+1}}=0, \text{ for } 1\leq \ell \leq m-1.$$ From the above relations, we see that $\{x_\ell\}$ forms a basis for $(\Z\Gamma_{p,p-q},Q_{p,p-q})^\perp_{\psi_{p,p-q}}$, where\begin{align*} 
x_1 &=  e_1 - e_2 + \cdots + (-1)^{a_1}e_{{n_2}-1}, \\ 
x_\ell &=  e_{n_\ell-1} - e_{n_\ell} + \cdots + (-1)^{a_\ell}e_{{n_{\ell+1}}-1}, \text{ for } 2\leq \ell \leq m-1, \\
x_m &=  e_{n_m-1} - e_{n_m} + \cdots + (-1)^{a_m}e_{n_m+a_m-2}. 
\end{align*}

Finally, it is straightforward to check that the matrix, denoted by $M$, that represents the intersection form of $(\Z\Gamma_{p,p-q},Q_{p,p-q})^\perp_{\psi_{p,p-q}}$ with respect to the basis $\{x_\ell\}$ coincides with the matrix, denoted by $M_{p,q}$, that represents $Q_{p,q}$ with respect to the obvious basis for $\Z\Gamma_{p,q}$. Note that we have $M=P^{\top}M_{p,q}P$ where $P$ is a matrix that represents $\psi_{p,q}$. This implies that $P$ is unimodular and $\psi_{p,q}$ is an isomorphism. Then the result follows from \eqref{eqn:first} and \eqref{eqn:second}. \end{proof}

%
%
Proposition~\ref{prop:ortho} is motivated by \cite[Proposition 4.1]{Aceto-Celoria-Park:2018-1}. By restricting to a smaller family of lens spaces, Proposition~\ref{prop:ortho} gives the same conclusion as \cite[Proposition 4.1]{Aceto-Celoria-Park:2018-1} with a weaker assumption. We are now ready to prove Theorem~\ref{thm:2}.
\begin{proof}[Proof of Theorem~\ref{thm:2}]Let $W$ be the closed $4$-manifold resulting from gluing $X$ and $P(p,p-q)$ along $L(p,q)$. Again, by Donaldson's theorem~\cite{Donaldson:1987-1} we have an embedding 
$$\psi\colon
(H_2(X;\mathbb{Z})/\tors,Q_X) \oplus (\Z\Gamma_{p,p-q},Q_{p,p-q}) \hookrightarrow (\mathbb{Z}^{b_2(X)+\rk(\Z\Gamma_{p,p-q})}, \Id).
$$ Let $\psi_{p,p-q}$ be the embedding obtained by restricting $\psi$ to $\Z\Gamma_{p,p-q}$. Further, by restricting $\psi$ to $H_2(X;\mathbb{Z})/\tors$ we have the following embedding $$\psi|_{H_2(X;\mathbb{Z})/\tors}\colon
(H_2(X;\mathbb{Z})/\tors,Q_X)\hookrightarrow (\Z\Gamma_{p,p-q},Q_{p,p-q})^\perp_{\psi_{p,p-q}}.$$ Then the first part of the theorem follows from Proposition~\ref{prop:lowerbound} and Proposition~\ref{prop:ortho}.

Suppose now that $m=b_2(X)$, then by Proposition~\ref{prop:ortho} we have $$(\Z\Gamma_{p,p-q},Q_{p,p-q})^\perp_{\psi_{p,p-q}} \cong (\Z\Gamma_{p,q},Q_{p,q}).$$ Let $M_X$ and $M_{p,q}$ be matrices that represent $Q_X$ and $Q_{p,q}$, respectively. Then $M_X=P^{\top}M_{p,q}P$ where $P$ is a matrix that represents $\psi|_{H_2(X;\mathbb{Z})/\tors}$. Recall that $M_X$ presents a subgroup of $H_1(L(p,q);\mathbb{Z})$ (see, for instance, \cite[Section~2]{Owens-Strle:2006-1}) and $M_{p,q}$ presents $H_1(L(p,q);\mathbb{Z})$. In particular, $\det(M_X) \leq p$ and $\det(M_{p,q}) = p$, which implies that $P$ is unimodular and concludes the proof.\end{proof}

\bibliographystyle{alpha}
\def\MR#1{}
\bibliography{bib}

\begin{thebibliography}{AGL17}

\bibitem[ACP18]{Aceto-Celoria-Park:2018-1}
Paolo Aceto, Daniele Celoria, and JungHwan Park.
\newblock Rational cobordisms and integral homology.
\newblock ar{X}iv:1811.01433, 2018.

\bibitem[AGL17]{Aceto-Golla-Larson:2017-1}
Paolo Aceto, Marco Golla, and Kyle Larson.
\newblock Embedding 3-manifolds in spin 4-manifolds.
\newblock {\em J. Topol.}, 10(2):301--323, 2017.

\bibitem[BB08]{Budney-Burton:2008-1}
Ryan Budney and Benjamin~A. Burton.
\newblock Embeddings of 3-manifolds in {$S^4$} from the point of view of the
  11-tetrahedron census.
\newblock ar{X}iv:0810.2346, 2008.

\bibitem[Boy93]{Boyer:1993-1}
Steven Boyer.
\newblock Realization of simply-connected {$4$}-manifolds with a given
  boundary.
\newblock {\em Comment. Math. Helv.}, 68(1):20--47, 1993.

\bibitem[CH81]{Casson-Harer:1981-1}
Andrew~J. Casson and John~L. Harer.
\newblock Some homology lens spaces which bound rational homology balls.
\newblock {\em Pacific J. Math.}, 96(1):23--36, 1981.

\bibitem[CH98]{Crisp-Hillman:1998-1}
John~S. Crisp and Jonathan~A. Hillman.
\newblock Embedding {S}eifert fibred {$3$}-manifolds and {${\rm
  Sol}^3$}-manifolds in {$4$}-space.
\newblock {\em Proc. London Math. Soc. (3)}, 76(3):685--710, 1998.

\bibitem[Don87]{Donaldson:1987-1}
Simon~K. Donaldson.
\newblock The orientation of {Y}ang-{M}ills moduli spaces and {$4$}-manifold
  topology.
\newblock {\em J. Differential Geom.}, 26(3):397--428, 1987.

\bibitem[Don15]{Donald:2015-1}
Andrew Donald.
\newblock Embedding {S}eifert manifolds in {$S^4$}.
\newblock {\em Trans. Amer. Math. Soc.}, 367(1):559--595, 2015.

\bibitem[Edm05]{Edmonds:2005-1}
Allan~L. Edmonds.
\newblock Homology lens spaces in topological 4-manifolds.
\newblock {\em Illinois J. Math.}, 49(3):827--837, 2005.

\bibitem[EL96]{Edmonds-Livingston:1996-1}
Allan~L. Edmonds and Charles Livingston.
\newblock Embedding punctured lens spaces in four-manifolds.
\newblock {\em Comment. Math. Helv.}, 71(2):169--191, 1996.

\bibitem[Eps65]{Epstein:1965-1}
David B.~A. Epstein.
\newblock Embedding punctured manifolds.
\newblock {\em Proc. Amer. Math. Soc.}, 16:175--176, 1965.

\bibitem[Fre82]{Freedman:1982-1}
Michael~H. Freedman.
\newblock The topology of four-dimensional manifolds.
\newblock {\em J. Differential Geom.}, 17(3):357--453, 1982.

\bibitem[FS87]{Fintushel-Stern:1987-1}
Ronald Fintushel and Ronald Stern.
\newblock Rational homology cobordisms of spherical space forms.
\newblock {\em Topology}, 26(3):385--393, 1987.

\bibitem[GL83]{Gilmer-Livingston:1983-1}
Patrick~M. Gilmer and Charles Livingston.
\newblock On embedding {$3$}-manifolds in {$4$}-space.
\newblock {\em Topology}, 22(3):241--252, 1983.

\bibitem[Han38]{Hantzsche:1938-1}
W.~Hantzsche.
\newblock Einlagerung von {M}annigfaltigkeiten in euklidische {R}\"{a}ume.
\newblock {\em Math. Z.}, 43(1):38--58, 1938.

\bibitem[Hil96]{Hillman:1996-1}
Jonathan~A. Hillman.
\newblock Embedding homology equivalent {$3$}-manifolds in {$4$}-space.
\newblock {\em Math. Z.}, 223(3):473--481, 1996.

\bibitem[Hil09]{Hillman:2009-1}
Jonathan~A. Hillman.
\newblock Embedding 3-manifolds with circle actions.
\newblock {\em Proc. Amer. Math. Soc.}, 137(12):4287--4294, 2009.

\bibitem[Hir61]{Hirsch:1961-1}
Morris~W. Hirsch.
\newblock The imbedding of bounding manifolds in euclidean space.
\newblock {\em Ann. of Math. (2)}, 74:494--497, 1961.

\bibitem[IM18]{Issa-McCoy:2018-1}
Ahmad Issa and Duncan McCoy.
\newblock Smoothly embedding seifert fibered spaces in {$S^4$}.
\newblock ar{X}iv:1810.04770, 2018.

\bibitem[Kaw77]{Kawauchi:1977-1}
Akio Kawauchi.
\newblock On quadratic forms of {$3$}-manifolds.
\newblock {\em Invent. Math.}, 43(2):177--198, 1977.

\bibitem[KK80]{Kawauchi-Kojima:1980-1}
Akio Kawauchi and Sadayoshi Kojima.
\newblock Algebraic classification of linking pairings on {$3$}-manifolds.
\newblock {\em Math. Ann.}, 253(1):29--42, 1980.

\bibitem[Lec12]{Lecuona:2012-1}
Ana~G. Lecuona.
\newblock On the slice-ribbon conjecture for {M}ontesinos knots.
\newblock {\em Trans. Amer. Math. Soc.}, 364(1):233--285, 2012.

\bibitem[Lis07]{Lisca:2007-1}
Paolo Lisca.
\newblock Lens spaces, rational balls and the ribbon conjecture.
\newblock {\em Geom. Topol.}, 11:429--472, 2007.

\bibitem[OS06]{Owens-Strle:2006-1}
Brendan Owens and Sa{\v{s}}o Strle.
\newblock Rational homology spheres and the four-ball genus of knots.
\newblock {\em Adv. Math.}, 200(1):196--216, 2006.

\bibitem[Rie74]{Riemenschneider:1974-1}
Oswald Riemenschneider.
\newblock Deformationen von {Q}uotientensingularit\"{a}ten (nach zyklischen
  {G}ruppen).
\newblock {\em Math. Ann.}, 209:211--248, 1974.

\bibitem[Roh65]{Rohlin:1965-1}
V.~A. Rohlin.
\newblock The embedding of non-orientable three-manifolds into five-dimensional
  {E}uclidean space.
\newblock {\em Dokl. Akad. Nauk SSSR}, 160:549--551, 1965.

\bibitem[Sto93]{Stong:1993-1}
Richard Stong.
\newblock Simply-connected {$4$}-manifolds with a given boundary.
\newblock {\em Topology Appl.}, 52(2):161--167, 1993.

\bibitem[Wal65]{Wall:1965-1}
C.~T.~C. Wall.
\newblock All {$3$}-manifolds imbed in {$5$}-space.
\newblock {\em Bull. Amer. Math. Soc.}, 71:564--567, 1965.

\bibitem[Zee65]{Zeeman:1965-1}
E.~C. Zeeman.
\newblock Twisting spun knots.
\newblock {\em Trans. Amer. Math. Soc.}, 115:471--495, 1965.

\end{thebibliography}
\end{document}